\documentclass[letterpaper, 10 pt, journal]{IEEEtran}

\usepackage{graphicx}
\usepackage{textcomp}
\usepackage{epsfig} 
\usepackage{epsfig,color,amsmath,cite}

\usepackage{amsthm} 
\usepackage{amsmath}    

\usepackage{enumitem} 
\setlist[itemize]{leftmargin=*}

\usepackage{bm}
\usepackage{epstopdf}
\usepackage{amssymb}
\usepackage{url}

\usepackage{multirow}
\usepackage{hhline}
\usepackage{booktabs}
\usepackage[linesnumbered,boxed,commentsnumbered,ruled,vlined,longend]{algorithm2e}
\usepackage{comment}
\newcommand{\eps}{\varepsilon}

\DeclareMathOperator{\trace}{trace}

\DeclareMathOperator*{\maximize}{maximize}

\DeclareMathOperator*{\subjectto}{subject\ to}

\makeatother
\DeclareMathAlphabet\mathbfcal{OMS}{cmsy}{b}{n}

\newtheorem{theorem}{Theorem}
\newtheorem{mydef}{Definition}
\newtheorem{mylem}{Lemma}

\newtheorem{myrem}{Remark}
\newtheorem{asmp}{Assumption}
\newtheorem{mycor}{{Corollary}}



\usepackage{stackengine}

\newcommand{\bmat}[1]{\begin{bmatrix} #1 \end{bmatrix}}

\newcommand{\m}{\boldsymbol}
\allowdisplaybreaks[4]
\usepackage[colorlinks = true,
linkcolor = blue,
urlcolor  = black,
citecolor = blue,
anchorcolor = blue]{hyperref}

\newcommand{\mbb}[1]{\mathbb{#1}}
\newcommand{\mr}[1]{\mathrm{#1}}
\usepackage[framemethod=TikZ]{mdframed}
\mdfdefinestyle{MyFrame}{%
	linecolor=black,
	outerlinewidth=1.25pt,
	roundcorner=1.25pt,
	innerrightmargin=5pt,
	innerleftmargin=5pt,}
	
\usepackage[noabbrev]{cleveref}

\usepackage{mathtools}

\DeclarePairedDelimiter\abs{\lvert}{\rvert}%
\DeclarePairedDelimiter\norm{\lVert}{\rVert}%

\makeatletter
\let\oldabs\abs
\def\abs{\@ifstar{\oldabs}{\oldabs*}}
\let\oldnorm\norm
\def\norm{\@ifstar{\oldnorm}{\oldnorm*}}
\makeatother

\usepackage[english]{babel}
\usepackage[utf8]{inputenc}
\usepackage[super]{nth}

\usepackage{graphicx}
\usepackage{float}
\usepackage[caption = false]{subfig}

\usepackage{array}
\usepackage{threeparttable}

\usepackage[english]{babel}
\usepackage[utf8]{inputenc}
\usepackage[super]{nth}

\usepackage{pifont}

\usepackage{mathtools}

\usepackage{mleftright}
\usepackage{xparse}

\NewDocumentCommand{\evaluat}{sO{\big}mm}{%
	\IfBooleanTF{#1}
	{\mleft. #3 \mright|_{#4}}
	{#3#2|_{#4}}%
}

\usepackage{lettrine} 

\usepackage{balance} 

\usepackage{tikz}
\usetikzlibrary{shapes.geometric, arrows.meta, decorations.pathmorphing, decorations.markings, calc, 3d}
\usetikzlibrary{positioning, fit, backgrounds}
\tikzset{
	tangent/.style={decoration={
			markings,mark=at position #1 with {
				\coordinate (ta) at (0,0);
				\coordinate (tb) at (0.1,0);
			}
		},postaction=decorate},
	tangent/.default=0.5
}

\newcommand{\parstart}[1]{\noindent \textbf{#1.}\;}

\newcommand{\Rn}[1]{\mathbb{R}^{#1}}

\newcommand{\titlebf}[1]{\title{\Large \vspace{0.7cm} \LARGE \centering {\textsc{\textbf{#1}}}}}

\usepackage{lettrine} 

\titlebf{Multilinear Extensions in Submodular Optimization for Optimal Sensor Scheduling in Nonlinear Networks}

\author{Mohamad H. Kazma and Ahmad F. Tah$\text{a}^{\diamond}$ 
\thanks{
	$^\diamond$Corresponding author. This work is supported by National Science Foundation under Grants 2152450 and 2151571. The authors are with the Civil $\&$ Environmental Engineering and Electrical $\&$ Computer Engineering Departments, Vanderbilt University, 2201 West End Ave, Nashville, Tennessee 37235. Emails: mohamad.h.kazma@vanderbilt.edu, ahmad.taha@vanderbilt.edu.}
}

\begin{document}
%

\fontdimen2\font=0.61ex
\fontdimen3\font=0.62ex

\maketitle
\thispagestyle{headings}

\pagestyle{headings}

\markboth{To Appear in the 2025 American Control Conference (ACC'2025), Denver, Colorado, July 2025}{}

\newcommand\reduline{%
	\bgroup\markoverwith
	{\textcolor{red}{\pgfsetfillopacity{0.2}\rule[-0.5ex]{2pt}{10pt}\pgfsetfillopacity{1}}%
		\textcolor{red}{\llap{\rule[0.4ex]{2pt}{0.4pt}}\llap{\rule[0.7ex]{2pt}{0.4pt}}}%
	}%
	\ULon}

\begin{abstract}
Optimal sensing nodes selection (SNS) in dynamic systems is a combinatorial optimization problem that has been thoroughly studied in the recent literature. This problem can be formulated within the context of set optimization. For high-dimensional nonlinear systems, the problem is extremely difficult to solve. It scales poorly too. Current literature poses combinatorial submodular set optimization problems via maximizing observability performance metrics subject to matroid constraints. Such an approach is typically solved using greedy algorithms that require lower computational effort yet often yield sub-optimal solutions.  In this paper, we address the SNS problem for nonlinear dynamical networks using a variational form of the system dynamics, that basically perturb the system physics. As a result, we show that the observability performance metrics under such system representation are indeed submodular. The optimal problem is then solved using the \textit{multilinear continuous extension}. This extension offers a computationally scalable and approximate continuous relaxation with a performance guarantee. The effectiveness of the extended submodular program is studied and compared to greedy algorithms. We demonstrate the proposed set optimization formulation for SNS on nonlinear natural gas combustion networks.
\end{abstract}
\vspace{-0.1cm}
\begin{IEEEkeywords}
	Sensing node selection, nonlinear dynamical networks, observability, submodularity, multilinear extension\vspace{-0.2cm}
\end{IEEEkeywords}
\vspace{-0.1cm}
\section{Introduction and Paper Contributions}\vspace{-0.1cm}\label{sec:introduction}

\noindent \lettrine[lines=2]{T}{he} sensing node selection (SNS) problem can be formulated within the context of \textit{set optimization}, where the decision variables are discrete sets and not vectors or matrices. One key property that is relevant to combinatorial set optimization is that of submodularity; it is a diminishing returns property which provides provable performance guarantees that allow to solve the combinatorial SNS problem via simple greedy algorithms. 
This requires posing the SNS problem as a \textit{submodular} set optimization problem with objective function ${f}$ in the form
\vspace{-0.1cm}
\begin{equation}\vspace{-0.1cm}\label{eq:sns_submod}
 f^*:=\maximize_{\mathcal{S}\subseteq\mathcal{V},\; \mathcal{S} \in \mathcal{I}} \;\; f(\mathcal{S}),
\end{equation}
where the set $\mathcal{V}$ represent a set of available sensor nodes to choose from and the set $\mathcal{I}$ represents \textit{matroid} constraints. A typical matroid constraint for SNS is the cardinality constraint $\mathcal{I}_c=\{ \mathcal{S} \subseteq \mathcal{V}: \abs{\mathcal{S}} = r\}$ for some $r$ that represents the feasible number of sensing nodes. Such an approach is typically solved for $f^*_{\mathcal{S}}$ using greedy algorithms that require lower computational effort while achieving a $(1-1/e)$ performance guarantee. Given that $(1-1/e)\approx 0.63$, $f^*_{\mathcal{S}}$ is at least $0.63$ the optimal value $f^*$. 

In this paper, we address solving the SNS problem for a nonlinear dynamical system via quantifying submodular observability-based metrics while posing it as a submodular maximization problem~\eqref{eq:sns_submod}. For linear systems, Gramian-based observability quantification allows for scalable SNS by exploiting submodular observability notions as demonstrated in~\cite{Summers2016}. As compared to nonlinear systems, observability-based SNS for nonlinear systems remains a topic of ongoing research~\cite{Haber2018}. Typically, quantifying nonlinear system observability can be approached by considering an empirical Gramian approach~\cite{Lall1999} or a Lie derivative approach~\cite{Krener1983}. Both methods can become infeasible for large scale systems when considering solving the SNS problem. A \textit{variational} approach can be considered to handle system nonlinearities. Such system is considered linear-time varying along the system trajectory and thus an observability Gramian can be computed more efficiently~\cite{Kazma2024}. 

Furthermore, the application of the aforementioned greedy algorithm is well-studied throughout the literature for the cardinality constrained set problem~\eqref{eq:sns_submod}. Recently, there has been growing interest in extending such algorithms to handle other matroid constraints. In particular, the multilinear relaxation offers a powerful continuous extension to handle various matroids and non-monotone set functions~\cite{Calinescu2011}. The application of this extension for submodular set maximization is introduced in~\cite{Vondrak2008a}; its applicability has gained interest within the literature~\cite{Calinescu2011,Chekuri2014,Chekuri2019}. The relaxed submodular problem can be solved using a \textit{continuous greedy }algorithm along with a \textit{pipage rounding} algorithm while achieving a $(1-1/e)$ performance guarantee. Several algorithms have been developed for this continuous view point, this includes parallelization algorithms that can extensively reduce the computational effort of the maximization problem~\cite{Chekuri2019}. Accordingly, in this paper we apply the multilinear extension to solve the optimal SNS problem. To the best the authors' knowledge, the application of such an extension has not been applied within the context of observability-based SNS in linear and nonlinear systems.

\parstart{Paper Contributions} The main contributions of this paper are as follows. $(i)$ We show that the observability Gramian which is based on the variational form of a nonlinear dynamical system is a modular set function under the parameterized sensor selection formulation. $(ii)$ We show that SNS problem that is based on observability metrics under the action of the variational Gramian are submodular and modular. In particular, we show that the $\trace$ is modular, and both the $\mr{rank}$ and $\mr{log}\,\mr{det}$ are monotone submodular. This is analogous to the case for a linear observability Gramian as demonstrated in~\cite{Summers2016}. $(iii)$ We introduce and demonstrate a multilinear continuous relaxation to the SNS maximization problem~\eqref{eq:sns_submod}. A continuous greedy algorithm along with a pipage rounding algorithm are used to solve the optimal problem under the cardinality constraint while achieving a worst case bound. The proposed method is illustrated on a combustion reaction network. Note that we consider a linear measurement model for brevity of exposition, while the method also applies to nonlinear measurement models.

\parstart{Broader Impacts} The multilinear continuous extension enables posing~\eqref{eq:sns_submod} under different constraints while solving using efficient algorithms. There are many other types of constraints besides the cardinality constraint. For example, Knapsack constraints can be considered in the context of SNS. That is, we can assign budget constraints to the cardinality set maximization problem as $\max _S f(S)$ s.t. $\sum_{s \in \mathcal{V}} c(s) \leq B$ where $B$ is a non-negative budget. The applicability of the multilinear extension for such problem given parameter $\eps>0$ results in $(1-1 / e-\epsilon)$ performance bounds; see\cite{Chekuri2019}.
Accordingly, such an extension enables guaranteed performance under various matroid constraints that can arise when considering SNS applications for nonlinear systems.

\parstart{Notation} Let $\mathbb{N}$, $\Rn{n}$, and $\Rn{n\times m}$ denote the set of natural numbers, real-valued row vectors with size of $n$, and $n$-by-$m$ real matrices respectively. The cardinality of a set $\mathcal{N}$ is denoted by $|\mathcal{N}|$. The operators $\mr{log}\,\mr{det}(\m{A})$ and $\mr{trace}(\m{A})$ return the logarithmic determinant and trace of matrix $\m{A}$. The operator $\{\m{x}_{i}\}_{i=1}^{\mr{N}}  \in \Rn{\mr{N}n}$ constructs a column vector that concatenates vectors $\m{x}_i \in \Rn{n}$ for all $i \in \{1,2,\ldots, \mr{N}\}$.

\parstart{Paper Organization} The paper is organized as follows:~Section~\ref{sec:prelims} provides preliminaries on nonlinear observability. Section~\ref{sec:submodularity} introduces submodular maximization and formulates the multilinear extension for the SNS problem. Section~\ref{sec:ObsMetrics} provides evidence regarding the submodularity of the observability measures. The numerical results are presented in Section~\ref{sec:casestudies}.
\vspace{-0.4cm}
\section{Preliminaries on Nonlinear Observability}\vspace{-0.1cm}\label{sec:prelims}
In this section, we present the dynamical system setup and present its variational representation. Subsequently, we introduce the concept of variational observability Gramian.
\vspace{-0.3cm}
\subsection{Nonlinear Dynamical System Setup}\vspace{-0.1cm}\label{sec:setup}

In this paper, we consider the following discrete-time nonlinear dynamical system with a parameterized measurement equation that represents the action of the system dynamics evolving on the smooth manifold $\mathcal{M}$.
\vspace{-0.1cm}
\begin{equation}\vspace{-0.1cm}\label{eq:model_DT}
		\m{x}_{k+1} = \m x_{k} +\tilde{\m f}(\m x_k), \;\;\;\;\; \quad
		\m y_k = \m{h}(\m{x}_k)= \tilde{\m{C}} \m{x}_{k},
\end{equation}
where the discretization period is denoted as $T > 0$ and $k\in\mbb{N}$ refers to the discrete-time index, such that vector $\m x_k = \m x(kT)  \in \mbb{R}^{n_x}$ denotes the system states evolving in $\mathcal{M}$. The vector $\m{y}_k \in \mbb{R}^{n_{y}}$ denotes the global output measurements. The nonlinear mapping function $\tilde{\m{f}}(\cdot):\mathcal{M} \rightarrow\mbb{R}^{n_x}$ and nonlinear mapping measurement function $\m{h}(\cdot) :\mathcal{M} \rightarrow\mbb{R}^{n_y}$ are smooth and at least twice continuously differentiable. The parameterized measurement matrix $\tilde{\m{C}} \in \Rn{n_y \times n_x}$ represents the mapping of output states under a configuration of the sensors.

There are a plethora of discretization methods that can be utilized to discretize nonlinear systems and the choice of discretization method relies on the desired accuracy of the discretization. In this paper, the \textit{implicit Runge-Kutta} (IRK) method 
is the chosen discretization method. This method allows for a wide-range of application to dynamical systems that have various degrees of stiffness. Note that studying different discretization methods is outside the scope of this paper. With that in mind, for the IRK method the nonlinear function $\tilde{\m{f}}(\cdot)$ is defined as $\tilde{\m{f}}(\cdot):= \tfrac{T}{4}\left(\m f(\m \zeta_{1,k+1})+3\m f(\m \zeta_{2,k+1})\right),$
%
where auxiliary state vectors $\m \zeta_{1,k+1},\m \zeta_{2,k+1}\in\mbb{R}^{n_x}$ are auxiliary for computing $\m x_{k+1}$ provided that $\m x_{k}$ is given. For brevity, we omit the full description of auxiliary vectors $\m \zeta_{1,k+1}$ and $\m \zeta_{2,k+1}$; interested readers are referred to~\cite[Section II]{Kazma2024}.
The following assumption asserts that $\m{x}_k$ belongs within a set $\mathcal{X}$.
\vspace{-0.2cm}
\begin{asmp}\label{assumption:compact}
	Let the compact set $\mathcal{\m{X}}$ contain the set of all feasible solutions of system~\eqref{eq:model_DT}, then the system trajectory remains in $\mathcal{\m{X}} \subseteq \mathcal{M}$ for any initial state $\m{x}_{0} \in \mathcal{\m{X}}_{0}$ and $k \geq 0$.\vspace{-0.2cm}
\end{asmp}
\vspace{-0.2cm}
\vspace{-0.3cm}
\subsection{Nonlinear Observability: A Variational Approach}\vspace{-0.1cm}\label{sec:VarGram}
There are a myriad of different approaches for observability-based SNS. 
One approach to evaluate a dynamical system's observability is to compute the empirical observability Gramian of the system~\cite{Qi2015,Kunwoo2023}. 
However, scaling the internal states and measurements while capturing the local variations of states from the Gramians eigenvalues is not clearly defined~\cite{Krener2009}. Lie derivative on the other hand are typically not considered for assessing nonlinear observability under the context of sensor selection for the following reasons. $(i)$ Computing Lie derivatives is expensive since it requires the calculation of higher-order derivatives~\cite{Whalen2015a}; $(ii)$ it results in a qualitative observability rank condition that is not suitable in the context of SNS~\cite{Krener1983, Krener2009}.  
Considering a nonlinear discrete-time system, a moving horizon observability approach is introduced in~\cite{Haber2018} and further developed in our previous work~\cite{Kazma2023f}; it offers a more tractable and robust solution than the empirical Gramian. However, the developed Gramian has no clear relation to the linear observability Gramian. Recently, observability Gramians based on a variational system representation of the nonlinear dynamics have been developed; see~\cite{Kawano2021,Kazma2024}. As claimed by~\cite{Kazma2024}, the Gramian extends the linear Gramian to nonlinear systems. 

Consider two nearby trajectories $\m{x}_{k}$ and $\m{x}_{k} + \m\delta{\m{x}_k}$ resulting from initial states $\m{x}_{0}$ and $\m{x}_{0} + \m\delta{\m{x}_0}$. Note that $\m \delta{\m{x}_0}\in \Rn{n_x}$ represents the infinitesimal perturbation $\eps>0$ to initial states $\m{x}_0$. Its exponential decay or growth for $k \in \{0\;, 1\; ,\;\cdots\;,\mr{N}-1\}$ is denoted by $\m \delta\m{x}_k\in \Rn{n_x}$. The discrete-time interval for the simulation is defined as $\mr{N}$. To that end, the variational form of the discrete-time system~\eqref{eq:model_DT} is as follows
\vspace{-0.1cm}
\begin{equation}\label{eq:DiscVarState}
		\m{\delta}\m{x}_{k+1} =
		\m{\Phi}_{0}^{k}(\m{x}_0) \m{\delta}\m{x}_0, \quad \quad 
		\m{\delta}\m{y}_{k} =\m{\Psi}_{0}^{k}(\m{x}_{0})\m{\delta}\m{x}_{0},
\end{equation}
where $\m{\Phi}_{0}^{k}(\m{x}_0):=\left(\m{I}_{n_x}  +\tfrac{\partial\tilde{\m{f}}( \m{x}_{k})}{\partial\m{x}_{k}}\right)\tfrac{\partial\m{x}_{k}}{\partial\m{x}_{0}} \in \Rn{n_x \times n_x}$ defines the variational mapping function, such that $\m{\Phi}_{0}^{0}(\m{x}_0)=\m{I}_{n_x}$. The identity matrix is denoted as $\m{I}_{n_x} \in \Rn{n_x\times n_x}$. Matrix $\m{\Psi}_{0}^{k}(\m{x}_{0}) :=  \tilde{\m{C}} \m{\Phi}_{0}^{k}(\m{x}_0)\in \mathbb{R}^{n_y\times n_x}$ denotes the parameterized variational measurement mapping function. Vector $\m{\delta}\m{y}_k  \in \mathbb{R}^{n_y}$ represents the is the measurement along the variational trajectory. Readers are referred to~\cite{Kawano2021,Kazma2024} for the derivation of the variational system~\eqref{eq:DiscVarState}. Moving forward, we remove the dependency of $\m{\Phi}_{0}^{k}(\m{x}_0)=\m{\Phi}_{0}^{k}$ and $\m{\Psi}_{0}^{k}(\m{x}_0)=\m{\Psi}_{0}^{k}$ on $\m{x}_0$. 
\begin{myrem}~\label{rmk:Phi}
	Notice that $\m{\Phi}_{0}^{k}$ represents the derivative of~\eqref{eq:model_DT} with respect to $\m{x}_0$ for $k \in \{0\;,1\;,\;\cdots\;,\mr{N}-1\}$. Thus, the matrix $\m{\Phi}_{0}^{k}$ requires the knowledge of $\m{x}_k$ for all $k$. As such, we can apply the chain rule to evaluate $\m{\Phi}_{0}^{k}$ as
\vspace{-0.2cm}
	\begin{equation}\label{eq:Phi}
		\m{\Phi}_{0}^{k}
		= \m{\Phi}_{k-1}^{k}  \m{\Phi}_{k-2}^{k-1}
		\; \cdots \; 	\m{\Phi}_{0}^{1}\m{\Phi}_{0}^{0} 
		= \prod^{i=k}_{1}\m{\Phi}^{i}_{i-1}.
	\end{equation}\vspace{-0.2cm}
\end{myrem}
\vspace{-0.2cm}

The variational observability Gramian for~\eqref{eq:DiscVarState} with parameterized measurement model around initial state $\m{x}_{0} \in \mathcal{\m{X}}_0$ satisfying assumptions~\ref{assumption:compact} can be written as
\vspace{-0.1cm}
\begin{equation}\vspace{-0.1cm}\label{eq:prop_obs_gram} 
	 \m{V}_{o}(\mathcal{S}) := {\m{\Psi}}^{\top}\m{\Psi} \in \mathbb{R}^{n_x \times n_x},
\end{equation}
where  $\m{\Psi} \in \mathbb{R}^{\mr{N}n_y \times n_x}$ denotes the observability matrix that concatenates the observations $\m{\delta}{\m{y}}_k$ over the observation horizon $\mr{N}$ for $k \in \{0,\;1,\;\cdots\;,\;\mr{N}-1\}$ and can be written as
\begin{equation}~\label{eq:prop_obs_matrix}
	\m{\Psi}:= \bmat{\m{\Psi}_{0}^{0}\;,\m{\Psi}_{0}^{1}\;,\m{\Psi}_{0}^{2}\;, \;\cdots\; , \m{\Psi}_{0}^{\mr{N}-1}}^{\top},
\end{equation}
where $\m{\Psi}_{0}^{k}$ is the mapping function defined in~\eqref{eq:DiscVarState}. 

Considering the SNS problem in~\eqref{eq:sns_submod}, the parameterized variational observability Gramian around an initial state ${\m x}_0$ for $\mathcal{S}\subseteq\mathcal{V}$ can be defined as~\eqref{eq:param_obs_gramian_set}. Readers are referred to~\cite{Kazma2024} for derivation of~\eqref{eq:param_obs_gramian_set}.
\vspace{-0.1cm}
\begin{align}
 \m{V}_{o}(\mathcal{S}) = \sum_{j\in \mathcal{S}} \left(\sum_{i=0}^{\mr{N}-1}
 \left(\m{\varphi}_{0}^{i}\right)^\top  \tilde{\m{c}}_j^\top 
 \tilde{\m{c}}_j \m{\varphi}_{0}^{i}\right),\vspace{-0.3cm}\label{eq:param_obs_gramian_set}
\end{align}
where $\m{\varphi}^{i}_{0}$ represents the column vectors of matrix $\m{\Psi}_{0}^{k}$.
The notation $j\in \mathcal{S}$ refers to each activated sensor in $\mathcal{S} \subseteq \mathcal{V}$, that is $\tilde{\m{c}}_j =1$ for a selected sensing node, and $\tilde{\m{c}}_j =0$ otherwise.
\vspace{-0.2cm}
\begin{theorem}\label{theo:obs_equiv_lin}
	The variational observability Gramian~\eqref{eq:prop_obs_gram} reduces to the linear observability Gramian, denoted by ${\m{W}}_{o}$, for a LTI system and is equivalent to the empirical Gramian when considering a general nonlinear system.
\end{theorem}\setlength{\textfloatsep}{0pt}
\vspace{-0.2cm}
\begin{proof}\label{proof:VarGramLin}
	Refer to~\cite[Corollary~$1$ and Theorem~$2$]{Kazma2024} for the proof.\vspace{-0.1cm}
\end{proof}\setlength{\textfloatsep}{0pt}\vspace{-0.3cm}
The above equivalence relation indicates that the nonlinear system~\eqref{eq:model_DT} is observable if and only if the variational measurement mapping function, $\m{\Psi}_{0}^{k}$, is full rank.
\vspace{-0.2cm}
\vspace{-0.2cm}
\section{Submodular Maximization \& its Extensions}\vspace{-0.1cm}\label{sec:submodularity}
In this section, we provide a brief review of submodularity, submodular set optimization, and present the multilinear continuous extension of the SNS problem. 
\vspace{-0.3cm}
\subsection{Submodularity and the Greedy Approach}\vspace{-0.1cm}\label{sec:mod_subod}
Consider a finite set $\mathcal{V}$ and the set of all its subsets given by $2^{\mathcal{V}}$. Let a set function $f({\mathcal{S}}): 2^{\mathcal{V}} \rightarrow \mathbb{R}$ for any $\mathcal{S} \subseteq \mathcal{V}$, then the modularity and submodularity of $f$ can be defined as
\begin{mydef}(\hspace{-0.012cm}\cite{Bach2010})
		\label{def:modular_submodular}
		The set function $f$ is defined as \textit{modular} if for any weight function $w:\mathcal{V}\rightarrow \mbb{R}$ and $\mathcal{S}\subseteq\mathcal{V}$, the following holds
		$ f(\mathcal{S}) = w(\emptyset) + \sum_{s\in\mathcal{S}} w(s),$
				and $f$ is considered \textit{ submodular} if for any $\mathcal{A},\mathcal{B}\subseteq\mathcal{V}$ and $\mathcal{A}\subseteq\mathcal{B}$, the following holds for all $s\notin\mathcal{B}$
				\begin{align}~\label{eq:submodular_def}
						f(\mathcal{A}\cup\{s\}) - f(\mathcal{A})\geq  f(\mathcal{B}\cup\{s\}) - f(\mathcal{B}).
					\end{align}
\end{mydef}
The above definition~\eqref{eq:submodular_def} shows the diminishing returns property of a submodular set function. For additional definitions refer to~\cite[Theorem~$2.2$]{Bilmes2022}. The set function $f$ is said to be supermodular if the reverse inequality in~\eqref{eq:submodular_def} holds true for all $s\notin\mathcal{B}$. Based on Definition~\ref{def:modular_submodular}, a function is said to modular if it is both super and submodular. We also say that a set function is normalized if $f(\emptyset)=0$. 
\vspace{-0.1cm}
\begin{mydef}(\hspace{-0.012cm}\cite{Bach2010})
		\label{def:monotone_increasing}
		Let $f: 2^{\mathcal{V}}\rightarrow \mbb{R}$ denote a set function. For any $\mathcal{A},\mathcal{B}\subseteq\mathcal{V}$, the set function is monotone increasing if for $\mathcal{A}\subseteq\mathcal{B}$ we have $	f(\mathcal{B})\geq 	f(\mathcal{A})$ and is monotone decreasing if for $\mathcal{A}\subseteq\mathcal{B}$ we have $f(\mathcal{A})\geq f(\mathcal{B})$.
\end{mydef}
\vspace{-0.1cm}
A submodular function is the discrete analog to concave functions. When a set function $f$ is submodular, monotone increasing and normalized, it can be referred to as a \textit{polymatroid} function~\cite{Bilmes2022}. Such functions are indicative of information, meaning that the functions tend to give a high value from a set $\mathcal{A}\subseteq \mathcal{V}$ that has a large amount of information and give a lower value to a set $\mathcal{C}\subseteq \mathcal{V}$ of equal cardinality to $\mathcal{A}$ but with a smaller amount of information.

That being said, for a chosen polymatroid function $f(\mathcal{S})$, in the context of the SNS problem~\eqref{eq:sns_submod}, the information gain is indicative of the measurement information from a subset of sensors $j \in \mathcal{S}$ regarding the full state-space of the system. For such a submodular function, a greedy algorithm~\cite[Algorithm 1]{Kazma2023f} offers a theoretical worst-case bound according to the following theorem.
\begin{theorem}(\hspace{-0.012cm}\cite{Nemhauser1978}) Let $f: 2^V \rightarrow \mathbb{R}$ be a polymatroid function, $f^*$ be the optimal solution of SNS problem~\eqref{eq:sns_submod} and $f^*_{\mathcal{S}}$ be the solution obtained while utilizing the greedy algorithm. Then, the following performance bound holds true 
\vspace{-0.2cm}
\begin{align*}
			f^*_{\mathcal{S}} -f(\emptyset) \geq \left(1-\frac{1}{e}\right)\left(f^*-f(\emptyset)\right), \quad \text{with}\; f(\emptyset) =0.
	\end{align*}
\end{theorem}
\vspace{-0.2cm}
We note that the above bound is theoretical, and generally a greedy approach performs better in practice. It has been shown that a $99\%$ accuracy can be achieved for actuator placement; see~\cite{Summers2016} and the many references that cite this work. 
\vspace{-0.4cm}
\subsection{Submodular Maximization via a Multilinear Extension}\vspace{-0.1cm}\label{sec:multilinear}
As an alternative to the greedy algorithm approach described in the aforementioned section, it is sometimes useful to solve the submodular set maximization problem continuously. This can be done by applying continuous extensions to a submodular function; that is, extending $f(\mathcal{S})$ to a function $F: [0,1]^{n} \rightarrow\mathbb{R}$ which agrees with $f(\mathcal{S})$ on the hypercube vertices~\cite{Bai2018}.
Extensions to submodular functions include: $(i)$ the Lovász extension~\cite{Lovasz1983} which is equivalent to the exact convex closure of $f(\mathcal{S})$ and the multilinear extension~\cite{Vondrak2008a} is equivalent to an approximate concave closure. 

For the purpose of submodular maximization, the multilinear extension is shown to be useful~\cite{Vondrak2008a}, whereas the Lovász extension is applicable for submodular minimization problems. The application of the multilinear relaxation in the context of observability-based SNS problem, whether for linear or nonlinear systems, has to the best of our knowledge, not been applied. For a submodular function $f: 2^{\mathcal{V}} \rightarrow$ $\mathbb{R}$, its multilinear extension $F:[0,1]^n \rightarrow \mathbb{R}$, where $n = |\mathcal{V}|$, in the continuous space can be written as
\vspace{-0.1cm}
\begin{equation}\vspace{-0.1cm}\label{eq:multilinext}
	\hspace{-0.3cm}F(\mathbf{x})\hspace{-0.05cm}=\hspace{-0.1cm}
	\sum_{\mathcal{S} \subset \mathcal{V}} f(\mathcal{S}) \prod_{s \in \mathcal{S}}[\mathbf{x}]_s \prod_{s \notin \mathcal{S}}\left(1-[\mathbf{x}]_s\right), \quad \mathbf{x} \in[0,1]^n .
\end{equation}

We define $\mathcal{S}_{\mathbf{x}}$ for any $\mathbf{x} \in[0,1]^n$ such that each element $s \in \mathcal{V}$ is included in $\mathcal{S}$ with probability $[\mathbf{x}]_s$ and not included with probability $1-[\mathbf{x}]_s$. As such, the notation $[\mathbf{x}]_s$ represents the $s$-th component of vector $\mathbf{x} \in [0,1]^n$. The multilinear extension $F(\mathbf{x})$ thus extends the function evaluation over the space between the vertices of the boolean hypercube $\{0,1\}^n$ to that of the vertices of hypercube $[0,1]^n$.

The computation of the multilinear extensions is not straightforward. That being said, the extension $F(\mathbf{x})$ for any submodular function $f({\mathcal{S}})$ can be approximated by randomly sampling sets $\mathcal{S}$ to the probabilities in $[\mathbf{x}]_s$~\cite{Gupta2009a}. With that in mind, $F(\mathbf{x})$ can be written as
\vspace{-0.1cm} 	
\begin{equation}\vspace{-0.1cm} \label{eq:expectedval}
	F(\mathbf{x})=\mathbb{E}\left[f\left(\mathcal{S}_{\mathbf{x}}\right)\right],
\end{equation}
where $\mathbb{E}[\cdot]$ denotes the expected value of $f\left(\mathcal{S}_{\mathbf{x}}\right)$. As such, by taking the derivatives of $F(\mathbf{x})$ we obtain the following
\vspace{-0.2cm}
\begin{equation}\vspace{-0.2cm}\label{eq:partialF}
\frac{\partial F(\mathbf{x})}{\partial[\mathbf{x}]_s}=\mathbb{E}\left[f\left(\mathcal{S}_{\mathbf{x}} \cup\{s\}\right)-f\left(\mathcal{S}_{\mathbf{x}} \backslash\{s\}\right)\right],
\end{equation}
and the second order derivative for and $a,b \in \mathcal{V}$ with $a\ne b$ can be written as
\begin{equation}
	\begin{split}
		\hspace{-0.4cm}	\frac{\partial^2 F(\mathbf{x})}{\partial[\mathbf{x}]_a \partial[\mathbf{x}]_b}=&\mathbb{E}\left[f\left(\mathcal{S}_{\mathbf{x}} \cup\{a, b\}\right)-f\left(\mathcal{S}_{\mathbf{x}} \cup\{b\} \backslash\{a\}\right)\right. \hspace{-0.3cm}\\
			&-f\left(\mathcal{S}_{\mathbf{x}} \cup\{a\} \backslash\{b\}\right)+f\left(\mathcal{S}_{\mathbf{x}} \backslash\{a, b\}\right).\hspace{-0.3cm}\label{eq:partialF2}
	\end{split}
\end{equation}

Readers are referred to~\cite{Vondrak2008a} for the derivation of the above partial derivatives. We note that, the partial derivative $\frac{\partial F(\mathbf{x})}{\partial[\mathbf{x}]_s}\geq0$ if and only if $f$ is monotone and $\frac{\partial^2 F(\mathbf{x})}{\partial[\mathbf{x}]_a \partial[\mathbf{x}]_b}\leq0$ if and only if $f$ is submodular. 

Considering the above, the submodular set maximization problem~\eqref{eq:sns_submod} under the application of the multilinear extension can be expressed as
\vspace{-0.2cm}
\begin{equation}\vspace{-0.2cm}\label{eq:sns_submod_extended}
	F^*_{\mathcal{S}} :=\maximize_{\mathcal{S}\subseteq\mathcal{V},\; \mathcal{S} \in \mathcal{I}_c} \;\;\;\; F(\mathbf{x}).
\end{equation}\setlength{\textfloatsep}{0pt}

Typically, $F(\mathbf{x})$ is concave in certain directions and convex in others, meaning that~\eqref{eq:sns_submod} is not easily solvable even under a simple cardinality constraint. In~\cite{Calinescu2011} a continuous greedy algorithm was developed to solve~\eqref{eq:sns_submod_extended}. The developed method solves for a fractional value of $F^*_{\mathcal{S}}$ and then utilizes a fractional rounding algorithm to transform this value into a discrete solution.

That being said, we consider the continuous greedy algorithm, detailed in~Algorithm~\ref{alg:continuous_greedy}, for solving~\eqref{eq:sns_submod_extended}. The algorithm defines a path $\mathbf{x}:[0,1] \rightarrow \mathcal{S}_{\mathbf{x}}$, where $\mathbf{x}(0)=0$ and $\mathbf{x}(1)$ are the output of the algorithm. Such that, $\mathbf{x}$ is defined by a differential equation, and the gradient of $\mathbf{x}$ is chosen greedily in $\mathcal{V}$ to maximize $F$, meaning that we maximize $\frac{\mathrm{d}}{\mathrm{d} l} \mathbf{x}(l) =\mr{argmax}_{\mathbf{x}\in \mathcal{S}_{\mathbf{x}}} \frac{\partial F}{\partial[\mathbf{x}]_s}(\mathbf{x}(l))$. This is equivalent to solving for
\vspace{-0.2cm}
\begin{align}\vspace{-0.2cm}
\underset{\mathcal{S} \in \mathcal{I}}{\mr{argmax}} \sum_{s \in \mathcal{S}} w_s \sim \mathbb{E}\left[f\left(\mathcal{S}_{\mathbf{x}} \cup\{s\}\right)-f\left(\mathcal{S}_{\mathbf{x}} \backslash\{s\}\right)\right],
\end{align}
 as a consequence of the equality defined in~\eqref{eq:partialF}.
  \begin{algorithm}[t]
 	\caption{{Continuous Greedy Algorithm}}\label{alg:continuous_greedy}
 	\DontPrintSemicolon
 	{\textbf{input:} multilinear extension $F$, ground set $\mathcal{V}$, $r$}\;
 	{\textbf{initialize: $\mathbf{x} \leftarrow \mathbf{0}$, $i \leftarrow 1$}}\;
 	\While{$i \leq r$}{
 		{\textbf{sample:} $K$ times of $\mathcal{S}$ from $\mathcal{V}$ according to $\mathbf{x}$}\;
 		\For{$s \in \mathcal{V}$}{
 			{\textbf{estimate:} $w_s \sim \mathbb{E}\left[f\left(\mathcal{S}_{\mathbf{x}} \cup\{s\}\right)-f\left(\mathcal{S}_{\mathbf{x}} \backslash\{s\}\right)\right]$}\;
 		}
 		{\textbf{solve for:} $\mathcal{S}^{\star}=\underset{\mathcal{S} \in \mathcal{I}}{\mr{argmax}} \sum_{s \in \mathcal{S}} w_s$}\;
 		{\textbf{update:} $\mathbf{x} \leftarrow \mathbf{x}_{\mathcal{S}^{\star}}$}\;
 		{$i \leftarrow i+1$}\;
 	}
 	{{Apply pipage rounding to transform the fractional solution $\mathbf{x}_{\mathcal{S}^{\star}}$ to a discrete solution.}\;}
 \end{algorithm}
 \setlength{\textfloatsep}{0pt}
 

The result obtained is fractional and thus a rounding algorithm is employed to convert this fractional solution. Randomly rounding the solution does not maintain the feasibility of the constraints, especially, equality constraints. A \textit{pipage rounding algorithm} has been shown to efficiently round the fractional value to a discrete value without any loss in the objective value. For brevity, we do not include the pipage rounding algorithm; refer to~\cite[Section 3.2]{Calinescu2011} for a formal description of the pipage rounding algorithm. The algorithm works by taking the fractional solution $\mathbf{x}_{\mathcal{S}^{\star}}$ 
obtained from the continuous greedy algorithm and then gradually eliminating all the fractional variables. This is done by minimizing along the convex direction of the multilinear set function. The result iterates until it agrees with the vertices of the hypercube $\{0,1\}^n$. The following Lemma~\ref{lemma:polytime} shows that a discrete solution can be solved in polynomial time.  

\vspace{-0.1cm}
\begin{mylem}(\hspace{-0.012cm}\cite{Calinescu2011}\label{lemma:polytime})
	Given $\mathbf{x}_{\mathcal{S}^{\star}}$, the pipage rounding algorithm results in a discrete solution $\mathcal{S} \in \mathcal{I}$ of value $\mathbf{E}[f(S)] \geq F\left(\mathbf{x}_{\mathcal{S}^{\star}}\right)$ in polynomial time.
\end{mylem}
\vspace{-0.1cm}
Note that, compared to the greedy algorithm with a running time complexity of $\mathcal{O}(|\mathcal{S}||\mathcal{V}|)$, the continuous greedy algorithm has a running time of $\mathcal{O}(|\mathcal{V}|^3 \log(|\mathcal{V}|))$~\cite{Calinescu2011}. The following theorem ensures a performance bound for solving~\eqref{eq:sns_submod_extended} via the continuous greedy algorithm.
\vspace{-0.1cm} 
\begin{theorem}(\hspace{-0.012cm}\cite{Calinescu2011})~\label{theo:guarenteeMulti}
	Let $f: 2^V \rightarrow \mathbb{R}$ be a polymatroid function and $F:[0,1]^n \rightarrow \mathbb{R}$ be its multilinear extension. Let $f^*$ denote the optimal solution of SNS problem~\eqref{eq:sns_submod} and $F^*_{\mathcal{S}}$ denote the solution computed using the continuous greedy algorithm. 
	Then, the following performance bound holds true
\vspace{-0.2cm}
	\begin{align*}
		F^*_{\mathcal{S}}-f(\emptyset) \geq \left(1-\frac{1}{e}\right)\left(f^*-f(\emptyset)\right), \quad \text{with}\; f(\emptyset) =0.
	\end{align*}
	 \vspace{-0.5cm}
\end{theorem} 
For both the presented submodular maximization frameworks the $1/ e$ guarantee holds true regardless of the size of the initial set $\mathcal{V}$ and which polymatroid function $f$ is being optimized. Given the aforementioned performance guarantees of the presented algorithms for submodular set maximization, the next section establishes the submodularity of certain observability measures. The observability measures are based on the parameterized variational observability Gramian~\eqref{eq:param_obs_gramian_set} for nonlinear systems.
\vspace{-0.4cm}
\section{Variational Gramian \& Submodularity}\vspace{-0.1cm}\label{sec:ObsMetrics}
In this section, we show that certain observability metrics that are based on the variational Gramian~\eqref{eq:prop_obs_gram} for nonlinear systems are indeed modular and submodular. 
For SNS applications, network measures based on system observability are often considered for quantifying information gain from the allocation of sensor nodes within a dynamical network. Observability-based network measures have key properties related to submodularity. Observability measures based on the linear observability Gramian have been shown to be submodular or modular. In particular, $\mr{log}\,\mr{det}$ and $\mr{rank}$ are submodular, while the $\mr{trace}$ is modular; see~\cite{Summers2016}. Nevertheless, other observability measures such as $\mr{log}\,\mr{det}\left(\m{W}_{o}^{-1}\right)$ and $\lambda_{\min}$ are non-submodular. Such important metrics have been shown to have provable guarantees when solved using greedy algorithms; refer to~\cite{Summers2019}.

In this paper, we consider only the metrics that have shown to have submodular properties in the linear case. With that in mind, we show that such properties hold true for nonlinear dynamical systems by considering such centrality measures under the action of the variational observability Gramian~\eqref{eq:param_obs_gramian_set}.

Accordingly, the following theorem establishes that the variational Gramian is a linear matrix function with respect to the selected sensing node $j\in \mathcal{S}$. Such, property shows that the Gramian can be computed from the sum of the individual contributions of each sensing node. Theorem~\ref{theo:VarGramModular_prop} is essential for the proofs related to the modularity or submodularity of the observability-based measures. 
\vspace{-0.1cm}
\begin{theorem}\label{theo:VarGramModular_prop}
	The parameterized variational observability Gramian $\m{V}_{o}(\mathcal{S})$ for $\mathcal{S}\subseteq\mathcal{V}$ is a modular set function.
\end{theorem}
\begin{proof}~\label{proof:proofModular}
	For any $\mathcal{S}\subseteq\mathcal{V}$, observe that
		\vspace{-0.2cm}
	\begin{align*}\vspace{-0.2cm} \label{proof:1}
		\m{V}_{o}({\mathcal{S}})
		\hspace{-0.05cm}=\hspace{-0.05cm}
		\sum_{j\in \mathcal{S}} \left\{{\m{\Phi}_{0}^{k}}^\top \hspace{-0.05cm} \tilde{\m c}_j^\top
		 \tilde{\m c}_j \m{\Phi}_{0}^{k}\right\}^{\mr{N}-1}_{k=0} =  \sum_{j\in \mathcal{S}} \m{V}_{o}(j),	\vspace{-0.2cm}
	\end{align*}
	where matrix $\m{V}_{o}(\mathcal{S})$ represents a linear function with respect to $\tilde{\m c}_j$; it thus satisfies modularity:
	$\m{V}_{o}(\mathcal{S})= \m{V}_{o}(\emptyset) + \sum_{j\in\mathcal{S}} \m{V}_{o}(j)$. 
\end{proof}
\vspace{-0.1cm}

We note that the proposed method can be applied to nonlinear measurement models; however, for brevity and succinctness in the exposition of proofs, we have utilized a linear measurement model. Based on Theorem~\ref{theo:VarGramModular_prop}, the following result shows that the $\mr{trace}$ of the variational Gramian~\eqref{eq:param_obs_gramian_set}  is modular set function. 
\vspace{-0.1cm}
\begin{mycor}\label{prs:trace_prop}
	Set function			
	$f:\hspace{-0.05cm}2^{\mathcal{V}}\hspace{-0.05cm}\rightarrow\hspace{-0.05cm}\mbb{R}$ characterized  by 
	\vspace{-0.2cm}
	\begin{equation}
		f(\mathcal{S})= \mr{trace}\left(\m{V}_{o}({\mathcal{S}})\right), 	\vspace{-0.2cm}\label{eq:trace_mod} 
	\end{equation}
	for $\mathcal{S}\subseteq\mathcal{V}$ is a modular set function.
\end{mycor}
\begin{proof}\label{proof:traceModular}
	For any $\mathcal{S}\subseteq\mathcal{V}$, we can write the following
	\vspace{-0.2cm}
	\begin{align*}
		\mr{trace}\left(\m{V}_{o}({\mathcal{S}})\right) 
		&=\mr{trace}\left(\sum_{j\in \mathcal{S}}\left(\sum_{i=0}^{\mr{N}-1}\left(\m{\varphi}_{0}^{i}\right)^\top  \tilde{\m{c}}_j^\top \tilde{\m{c}}_j \m{\varphi}_{0}^{i}\right) \hspace{-0.05cm}\right), \\
		& = \sum_{j\in \mathcal{S}}  \left(\mr{trace}\left(\sum_{i=0}^{\mr{N}-1}\left(\m{\varphi}_{0}^{i}\right)^\top  \tilde{\m{c}}_j^\top \tilde{\m{c}}_j \m{\varphi}_{0}^{i}\right) \hspace{-0.05cm}\right) ,\\
		&=\sum_{j\in \mathcal{S}} 
		 \left(\mr{trace}\left(\m{V}_{o}(j)\right)\right),
	\end{align*}	
	where the last equality is a due to the modularity of the parameterized variational Gramian. This shows that the $\mr{trace}(\m{V}_o(\mathcal{S}))$ can be written as a linear matrix function and thus a modular set function. 
\end{proof}
The following result shows that the $\mr{rank}$ of the variational Gramian~\eqref{eq:param_obs_gramian_set} is submodular and monotone increasing. 
\vspace{-0.1cm}
\begin{mycor}\label{prs:rank_prop}
	Set function
	$f:\hspace{-0.05cm}2^{\mathcal{V}}\hspace{-0.05cm}\rightarrow\hspace{-0.05cm}\mbb{R}$ defined as
	\vspace{-0.2cm} 	
	\begin{equation}\vspace{-0.2cm} \label{eq:rank_submod} 
	f(\mathcal{S})= \mr{rank}\left(\m{V}_{o}({\mathcal{S}})\right), 
	\end{equation}
	for $\mathcal{S}\subseteq\mathcal{V}$ is a submodular monotone increasing set function.
\end{mycor}
 \setlength{\textfloatsep}{0pt}
\vspace{-0.2cm}
\begin{proof}\label{proof:ranksubModular}
	For any $\mathcal{S}\subseteq\mathcal{V}$, first we demonstrate that $f(\mathcal{S})$ in \eqref{eq:rank_submod} is submodular. We define the derived set function $f_s: 2^{\mathcal{V}\setminus\{s\}}\rightarrow \mbb{R}$ for a $a \in \mathcal{V}$ as
	\begin{align*}
		f_s(\mathcal{S}) &\,= f\left(\mathcal{S} \cup \{s\}\right)  - f\left(\{s\}\right), \\[-0.2em] 
		&\,= 
		\mr{rank}\left(\hspace{-0.05cm}\m{V}_{o}(\mathcal{S} \cup \{s\})\hspace{-0.05cm}\right) \hspace{-0.05cm}- \hspace{-0.05cm}\mr{rank}\left(\hspace{-0.05cm}\m{V}_{o}(\{s\})\hspace{-0.05cm}\right), \\[-0.2em] 
		&\,=
		\mr{rank}\left(\hspace{-0.05cm}\m{V}_{o}(\mathcal{S}) + \m{V}_{o}(\{s\})\hspace{-0.05cm}\right)\hspace{-0.05cm}-\hspace{-0.05cm}\mr{rank}\left(\hspace{-0.05cm}\m{V}_{o}(\{s\})\hspace{-0.05cm}\right), \\[-0.2em] 
		&\,=
		\mr{rank}\left(\hspace{-0.05cm}\m{V}_{o}(\{s\})\hspace{-0.05cm}\right) \\[-0.2em] 
		&\;\;\;\,\;\;-\mr{dim}\left(\mr{image}\left(\m{V}_{o}(\mathcal{S})\right) \cap \mr{image}\left(\m{V}_{o}(\{s\})\right)\hspace{-0.05cm}\right),
	\end{align*}
	This indicates that  $f_a(\cdot)$ is monotone decreasing since $\mr{rank}\left(\hspace{-0.05cm}\m{V}_{o}\left(\{a\}\right)\hspace{-0.05cm}\right)$ is constant while the dimension of $\mr{image}\left(\m{V}_{o}(\mathcal{S})\right)$ is increasing with $\mathcal{S}$. This implies that $f(\cdot)$ in \eqref{eq:rank_submod} is submodular \cite{Summers2016,Lovasz1983}. Second, it is straightforward to show that $f(\cdot)$ is also monotone increasing since for $\mathcal{A}\subseteq\mathcal{B}$ provided that $\mathcal{A},\mathcal{B}\subseteq\mathcal{V}$ implies $f(\mathcal{B})\geq f(\mathcal{A})$. 
\end{proof}
\vspace{-0.2cm}

The following result provides evidence that the $\mr{log}\,\mr{det}$ of the variational Gramian~\eqref{eq:param_obs_gramian_set} is submodular and monotone increasing. 
\vspace{-0.2cm}
\begin{mycor}\label{prs:Trace-Logdet_prop}
	Set function $f:2^{\mathcal{V}}\rightarrow\hspace{-0.05cm}\mbb{R}$ defined as
	\vspace{-0.2cm} 	
	\begin{equation}\vspace{-0.2cm} \label{eq:logdet_submodular} 
		f(\mathcal{S}):=\mr{log}\,\mr{det}\left({\m V}(\mathcal{S})\right),
	\end{equation} 
	for $\mathcal{S}\subseteq\mathcal{V}$ is a submodular monotone increasing set function.
\end{mycor}
\vspace{-0.2cm}
\begin{proof}\label{proof:submodular-log-det}
	Let $f_s: 2^{V-\{s\}} \rightarrow \mathbb{R}$ denote a derived set function defined as
	\begin{align*}
		f_s(\mathcal{S}) & =\mr{log}\,\mr{det} \m{V}_o\left({\mathcal{S} \cup\{s\}}\right)-\mr{log}\,\mr{det} \m{V}_{o}(\mathcal{S}),\\
		& =\mr{log}\,\mr{det}\left(\m{V}_{o}(\mathcal{S})+\m{V}_{o}(\{s\})\right)-\mr{log}\,\mr{det} \m{V}_{o}(\mathcal{S}).
	\end{align*}	
	We first show $f_s(\mathcal{S})$ is monotone decreasing for any $s \in V$. That being said, let $\mathcal{A} \subseteq \mathcal{B}\subseteq \mathcal{V}-\{s\}$, and let $\m{V}_o(\tilde{\m{c}}) = \m{V}_o(\mathcal{A}) +\tilde{\m{c}}\left( \m{V}_o(\mathcal{B})- \m{V}_o(\mathcal{A})\right)$ for $\tilde{\m{c}}\in[0,1]$. Then for 
	\begin{align*}
		\tilde{f}_s(\m{V}_o(\tilde{\m{c}}))=\mr{log}\,\mr{det}\left(\m{V}_o(\tilde{\m{c}})+\m{V}_{o}(\mathcal{S})\right)-\mr{log}\,\mr{det}\left(\m{V}_o(\tilde{\m{c}})\right),
	\end{align*}
	we obtain the following
	\vspace{-0.2cm}
	\begin{align*}\vspace{-0.2cm}\label{eq:proof:cor3}
		\frac{\mr{d}}{\mr{d} \tilde{\m{c}}} \tilde{f}_s(\m{V}_o(\tilde{\m{c}}))
	\vspace{-0.2cm}	&=\vspace{-0.2cm}
		\mr{trace}\Big[
		\left(\left(\m{V}_o(\tilde{\m{c}})+\m{V}_{o}(\mathcal{S})\right)^{-1}-\m{V}_o(\tilde{\m{c}})^{-1}\right)\\
		\vspace{-0.2cm}&\quad\quad\quad\quad\;
		\left(\m{V}_o(\mathcal{B})-\m{V}_o(\mathcal{A})\right) \Big] \leq 0.
	\end{align*}
	Such that $	\left(\left(\m{V}_o(\tilde{\m{c}})+\m{V}_{o}(\mathcal{S})\right)^{-1}-\m{V}_o(\tilde{\m{c}})^{-1}\right)^{-1} \preceq 0$, and $\left(\m{V}_o(\mathcal{B})-\m{V}_o(\mathcal{A})\right)\succeq 0$, then the above inequality holds. We have therefore shown that $f(\mathcal{S})$ is submodular and ${f}_s$ is monotone decreasing. Then, by the additive property of $\m{V}_{o}(\mathcal{S})$ (see~\cite{Summers2016}) we have $f(\mathcal{S})$ being monotone increasing. The proof is analogous to~\cite[Theorem 6]{Summers2016} and~\cite[Lemma 3]{Zhou2019}
\end{proof}
\vspace{-0.2cm}
\begin{myrem}\label{rmk:positivedefinite}
	Notice that, for the $\mr{log}\,\mr{det}$ to be submodular and monotone increasing, the variational observability Gramian can have zero eigenvalues.
\end{myrem}
\vspace{-0.1cm}
In study~\cite{Zhou2019}, the considered observability measures are based on the Lie derivative matrix $\m{O}_{l}$. That being said, the above submodular properties hold true if and only if $\m{O}_{l}$ is full rank. In this case when considering the variational Gramian, there is no such restriction. The submodularity of the $\mr{log}\,\mr{det}$ still holds in rank deficient situations. Such situations can arise when not enough sensing nodes are chosen and thereby the system is not yet fully observable. The ensuing section demonstrates the validity of the SNS problem under the action of the studied variational observability measures.

\vspace{-0.3cm}
\section{Numerical Results}\vspace{-0.1cm}\label{sec:casestudies} 
In this paper, we consider a nonlinear system that represents a  natural gas combustion reaction network written as 
\vspace{-0.2cm}
\begin{equation}\vspace{-0.2cm}\label{eq:combustion_network}
	\dot{\m {x}}(t)=\Theta \boldsymbol{\psi}\left(\m {x}(t)\right),
\end{equation}
such that, the polynomial functions of concentrations $\psi_{j}$ $j=\{1,\;2,\;\ldots\;, \; N_{r}\}$ are concatenated in vector $\boldsymbol{\psi}\left(\m {x}\right)=[\psi_{1}\left(\m {x} \right),\psi_{2}\left(\m {x} \right),\ldots, \psi_{n_{r}}\left(\m {x} \right)]^{T}$. The concentrations of $n_x$ chemical species are denoted by vector $\m {x}=[x_{1},\; x_{2},\;\cdots\;,\; x_{n_{x}}]$. The stoichiometric coefficients $q_{ji}$ and $w_{ji}$ are defined by constant matrix $\Theta=[w_{ji}-q_{ji} ]\in \mathbb{R}^{n_{x}\times N_{r}}$. We denote the number of chemical reactions by $N_{r}$.
The considered network is a natural gas combustion reaction network $\mr{GRI30}$ that has $n_x=53$ chemical species and $N_r = 325$ reactions. We refer the readers to~\cite[Section $\mr{V}$]{Haber2018} for specifics regarding system parameters. Based on analyzing the system's initial state response, the chosen observation window is $\mr{N} = 1000$ and the discretization constant is set at $T=1\cdot 10^{-12}$. 
The actual initial state $\m{x}_0=[0, 0, 0, 2, \cdots, 1, \cdots, 7.52, \cdots, 0]$. 
The system is simulated after applying a system disturbance based on $\m{x}_{0} = \left(1+ \alpha_{d}\right)\m{x}_{0}$ where $\alpha_{d}\in\mathbb{R}$ is the disturbance magnitude between $(0,0.2)$. 

We assess the applicability and validity of the extended SNS problem~\eqref{eq:sns_submod} by comparing the state estimation error based on measurements from the optimally selected nodes to that obtained from the greedy algorithm. The optimality of the selected sensor nodes is directly related to the state estimation error. This is due to the underlying relations between observability and the ability to infer system states from limited measurement data. Let $\mathcal{S}^*_{M}$ define the optimal sensor node location resulting from solving~\eqref{eq:sns_submod_extended} using 
 the continuous greedy algorithm and let $\mathcal{S}^*$  define the optimal sensor node location resulting from solving~\eqref{eq:sns_submod} using the greedy algorithm. That being said, let $\m{x}_{\text{actual}}$ denote the state estimate resulting from solving the nonlinear state estimation problem expressed as
\begin{equation}\label{eq:opt_SE}
\hspace{-0.3cm}	\underset{\tilde{\m x}_0\in \mathbfcal{X}_0}{\mr{minimize}}\;\; {\m g(\tilde{\m x}_0)}^\top {\m g(\tilde{\m x}_0)},\quad 
	\subjectto \;\; {\tilde{\m x}}^l_0 \leq \tilde{\m x}_0 \leq {\tilde{\m x}}^u_0,\hspace{-0.3cm}
\end{equation}
where ${\tilde{\m x}}^l_0$ and ${\tilde{\m x}}^u_0$ are the lower and upper bounds of $\tilde{\m x}_0$. The vector function $\m{g}(\cdot):\mbb{\m R}^{n_x}\rightarrow\mbb{\m R}^{\mr{N} n_y}$ that is defined as $\m{g}(\tilde{\m{x}}_0):= \tilde{\m{y}} - \tilde{\m C} \tilde{\m{x}}$ represents the open-lifted observer. The measurement vector $\tilde{\m y} := \{\tilde{\m y}_{i}\}_{i=1}^{\mr{N}-1}\in \mbb{R}^{\mr{N}n_y}$ and estimated state-vector $\tilde{\m{x}}:=\{\tilde{\m x}_{i}\}_{i=1}^{\mr{N}-1}\in \mbb{R}^{\mr{N}n_x}$. The above least squares optimization problem is based on an open observer framework introduced in~\cite{Haber2018}. It is solved using the trust-region-reflective algorithm on MATLAB. The state estimation error can be written as $\left\|\m {x}_{\text{actual}}-\tilde{\m {x}} \right\|_{2}/\left\|\m {x}_{\text{actual}}  \right\|_{2}$. This validates the effectiveness of the solution obtained from~\eqref{eq:sns_submod_extended}, thereby achieving the performance bound as indicated in Theorem~\ref{theo:guarenteeMulti}. 

\begin{figure}[t]
	\centering
	\includegraphics[keepaspectratio=true,scale=0.53]{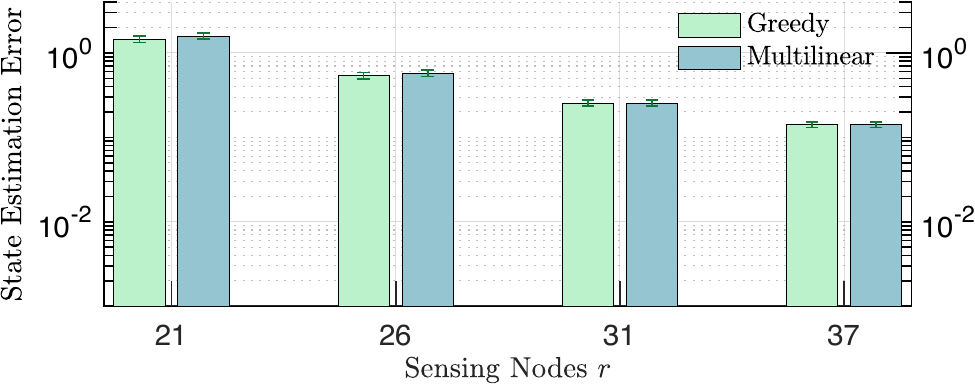 }
	\vspace{-0.38cm}
	\caption{State estimation error based on the optimal selected node obtained from greedy algorithm (left) and continuous greedy algorithm (right).}\label{fig:EstErr}
\end{figure}

The state estimation errors resulting from solving the SNS problems based on the aforementioned algorithms is depicted in Figure~\ref{fig:EstErr}. The maximization problems are solved for sensor nodes cardinality constraint $|\mathcal{S}|=r$ with $r= [21,26,31,37]$. The SNS problem is solved for both methods by considering $20$ generated simulations based on $\m{x}_0$ chosen randomly by applying perturbation $\alpha_{d}$. The results show that the optimal solution $\mathcal{S}^*_{M}$ yields similar state estimation values when compared with the estimation values of optimal solution $\mathcal{S}^*$. Consequently, the application of a multilinear extension for observability-based SNS can be further investigated under different matroid constraints; see Section~\ref{sec:introduction}. On such note, we conclude this section.


\balance
\vspace{-0.3cm}
\bibliographystyle{IEEEtran}
\bibliography{library}
\vspace{-0.5cm}

\end{document}